\documentclass[12pt]{amsart}
\usepackage{amsmath,amssymb,latexsym, amsthm, amscd, mathrsfs, stmaryrd}
\usepackage[all]{xy}

\theoremstyle{definition}
\theoremstyle{plain}
\newtheorem{prop}[subsubsection]{Proposition}
\newtheorem{thm}[subsubsection]{Theorem}
\newtheorem{lem}[subsubsection]{Lemma}
\newtheorem{cor}[subsubsection]{Corollary}

\newcommand{\Hom}{\mathrm{Hom}}
\newcommand{\mbf}{\mathbf}
\newcommand{\mbb}{\mathbb}
\newcommand{\mrm}{\mathrm}

\newcommand{\D}{\mathcal D}
\newcommand{\E}{\mbf E}
\newcommand{\F}{\mathcal F}
\newcommand{\G}{\mbf G_{\nu}}
\newcommand{\K}{\mbb K}

\newcommand{\Z}{\mbf Z}
\newcommand{\Fv}{\mathcal F_{\underset{\to}{\nu}} }
\newcommand{\Nv}{\mathcal N_{\underset{\to}{\nu}} }
\newcommand{\tF}{\widetilde{\mathcal F}}
\newcommand{\tFv}{\widetilde{\mathcal F}_{\underset{\to}{\nu}} }
\newcommand{\tGv}{\widetilde{\mathcal G}_{\underset{\to}{\nu}} }
\newcommand{\tNv}{\widetilde{\mathcal N}_{\underset{\to}{\nu}} }

\newcommand{\IC}{\mrm{IC}}

\newcommand{\End}{\mbox{End}}

\title[Perverse sheaves, representation varieties, framed Jordan quiver]{A class of perverse sheaves on framed  representation  varieties  of the  Jordan quiver }

\author{  Yiqiang Li}
\address{Department of Mathematics\\   State University of New York at Buffalo\\
244 Mathematics Building, Buffalo, NY 14260}
\email{yiqiang@buffalo.edu}

\date{\today}

\begin{document}
\begin{abstract}
A class of perverse sheaves  on  framed representation varieties of the Jordan quiver is defined and studied. Its relationship with product of symmetric groups, tensor product of Schur algebras, and tensor product of Fock spaces are addressed. 
\end{abstract}

\maketitle

\section{Introduction}

This article is devoted to the study of the geometry of framed nilpotent representation variety of Jordan quiver. There are several motivations indicating that the above geometry  is interesting.
Among others, the most relevant results are listed as follows.
\begin{enumerate}
\item The Springer correspondence of type $A$ in ~\cite{BM83}.
\item Ginzburg's Lagrangian  construction of Weyl group algebra and Schur algebra of type $A$ in ~\cite{CG97}.
\item Lusztig's study of the geometry of nilpotent orbits of type $A$ in ~\cite{L81}.
\item Achar-Henderson's study of the geometry of enhanced nilpotent orbits of type $A$ in ~\cite{AH08}.
\item The author's geometric realization of the tensor product of Verma module and simple module of a quantized enveloping algebra in ~\cite{Li11}.
\item Stroppel and Webster's recent work on $q$-Fock space in ~\cite{SW11}.
\item Grojnowski and Nakajima's study of Hilbert scheme of $n$ points on $\mbb C^2$ in ~\cite{N99} and ~\cite{G96}.
\end{enumerate}

The main results we obtain in this paper can be thought of as a generalization of the above results (1)-(2) to the ``tensor product case''.  More precisely, we define a class of  resolutions of singularities, 
$\pi_{\underline 1} : \tF_{\underline 1} \to \E_{\nu, d, Q}$,  on the framed nilpotent representation varieties, by mimicking the one studied in ~\cite{Li11}. 
We  show that these resolutions are semismall maps so that the push forward of the intersection cohomology complex of $  \tF_{\underline 1}$ to $\E_{\nu, d, Q}$ is again a semisimple perverse sheaf, 
say $L_{\underline 1}$. Furthermore, we show the following main results in this paper.
\begin{enumerate}
\item [($1'$)] The simple perverse sheaves appeared in $L_{\underline 1}$ are all intersection cohomology complexes (with trivial local systems)  and whose supports can be described explicitly in Theorem ~\ref{L-fine}.

\item [($2'$)] The endomorphism ring of $L_{\underline 1}$ in the abelian category of semisimple perverse sheaves on $\E_{\nu, d, Q}$ is isomorphic to the group algebra of a product of symmetric groups.

\item [($3'$)]  The singular supports of the intersection cohomology complexes in $L_{\underline 1}$ form a lagrangian subvariety in the cotangent bundle of $\E_{\nu, d,Q}$, whose subvariety of stable points , after taking certain GIT quotient, is the tensor product variety of the lagrangian variety defined by Grojnowski in ~\cite[\S 3]{G96}.

\item [($4'$)] A second way to state  ($2'$) is that the top Borel-Moore homology of the Steinberg-type variety $ Y_{\underline 1}= \tF_{\underline 1} \times_{\E_{\nu, d, Q}}  \tF_{\underline 1}$ is isomorphic to the group algebra of a product of symmetric groups. Moreover, the irreducible components of largest dimension in $Y_{\underline 1}$ form a basis of the above group algebra in Theorem ~\ref{Y}. We suspect that this basis  is  the naive tensor product of the Borel-Moore homology basis constructed in ~\cite{CG97}.

\item [($5'$)] A slight modification of ($4'$) gives rise to a construction of tensor product of Schur algebra of type $A$ as well. 

\item[($6'$)] The intersection cohomology complexes all together for the various $\nu$ are considered. 
It is shown that  there is a Heisenberg action on the space spanned by the intersection cohomology complexes, isomorphic to tensor product of Fock spaces.
\end{enumerate}

We emphasis that in the simplest nontrivial case (for the parameter $d=1$), the above results are all known in the literature 
such as ~\cite{AH08}, ~\cite{T09}, ~\cite{FGT09} and ~\cite{FG10},  except   ($3'$), ($5'$)  and the Heisenberg action in ($6'$),

In summary, the results ($1'$), ($2'$),  ($4'$) and ($5'$)  are generalizations of (1) and (2), the result ($3'$) is closely related to ~\cite{N99} and ~\cite{G96} 
and the result ($6'$) is closely related to ~\cite{SW11} and ~\cite{Li11}. It will be interesting to develop a theory similar to the results in ~\cite{L81} and ~\cite{AH08} on the local information of the intersection cohomology complexes defined in this paper.

\tableofcontents

\section{Preliminary}

\subsection{Semismall map}

 We recall  the semismall map and its variants from ~\cite{GM83}, ~\cite{BM83} and ~\cite{CG97}.
Let $\mu: M\to N$ be a proper map with $M$ and $N$ irreducible varieties over the field $\mbb C$ of complex numbers.  
Suppose that $N = \sqcup  N_{\alpha}$
is an algebraic stratification such that the restriction map $\mu: \mu^{-1}(N_{\alpha}) \to N_{\alpha}$ is a locally trivial fibration.  
We say that $\mu$ is $semismall$ if for any $\alpha$,
\[
2 \dim \mu^{-1}(x) \leq \dim M - \dim N_{\alpha}, \quad x\in N_{\alpha}.
\]
A semismall map is called $strictly$ $semismall$ if the above inequality is an equality for any $\alpha$.
A semismall map is called $small$ if the above inequality is an equality only when $N_{\alpha}$ is open dense.
A small map is a $small$ $resolution$ if it is also a resolution of singularities. 

Consider the fiber product $M\times_N M$, it admits a stratification induced from that of $N$:
\[
M\times_N M = \sqcup_{\alpha} M_{\alpha} \times_N M_{\alpha}, \quad M_{\alpha} = \mu^{-1}(N_{\alpha}).
\]
If the map $\mu$ is a semismall map, we have
\begin{align*}
\dim M\times_N M & =\mbox{sup}_{\alpha}  \dim M_{\alpha} \times_N M_{\alpha} 
\leq  \mbox{sup}_{\alpha} \dim M_{\alpha} + \dim \mu^{-1} (x_{\alpha}) \\
&=\mbox{sup}_{\alpha}  \dim N_{\alpha} + 2 \dim \mu^{-1}(x_{\alpha})  \leq \dim M,
\end{align*}
where $x\in N_{\alpha}$. 
It is clear that $\dim M\times_N M \geq \dim M$.
This shows that   if $\mu$ is semismall, then $\dim M\times_N M = \dim M$.

If $\dim M\times_N M=\dim M$, we have 
\[
\dim M_{\alpha}\times_N M_{\alpha}   = \dim N_{\alpha} + 2\dim \mu^{-1}(x_{\alpha})  \leq \dim M.
\]

Therefore, we see that $\mu$ is semismall if and only if 
\begin{align}
\label{global}
\dim M\times_N M =\dim M.
\end{align}
Furthermore,  if $\mu $ is semismall, then $\mu$ is small if and only if  $\dim M_{\alpha}\times_N M_{\alpha}   < \dim M$ for any  non-open-dense stratum $N_{\alpha}$. 

\subsection{Notation} 

We fix some notations. Let $\nu$ be a nonnegative integer, we denote $S_{\nu}$ to be the permutation group of $\nu$ letters. We identify $S_{\nu}$ with the set of permutation matrices of rank $\nu$.
Let $\mbb C[S_{\nu}]$ be the group algebra of $S_{\nu}$. We use similar notation to denote the group algebra of products of $S_{\nu}$.

Let $\lambda$ be a partition of $\nu$, we denote by $\lambda^{\perp}$ its dual partition. We denote by $V_{\lambda}$ the irreducible representation of $S_{\nu}$.

\subsection{Heisenberg algebra}
\label{H1}
In this section, we should recall the representation theory of the infinite-dimensional Heisenberg algebra $\mathfrak s$ from ~\cite[9.13]{K90}.
By definition, $\mathfrak s$ is  a Lie  algebra over $\mbb C$ with a  basis $p_i$, $q_i$ and $c$ for $i=1, 2, \cdots$, subject to the following relations:
\[
[p_i, q_j ]=\delta_{ij} c,\quad  \mbox{all other brackets are zero}.
\]
For a given $d\in\mbb C^{\times}$, the  Lie algebra $\mathfrak s$ has an irreducible module $\mbb F(d)$ on the Fock space
\begin{align}
\label{fock}
\mbb F = \mbb C [x_1, x_2, \cdots]. 
\end{align}
defined by 
\[
p_i (f) = d \frac{\partial}{\partial x_i} (f), \quad q_i (f) = x_i f,\quad c(f) =d f,\quad \forall f\in \mbb F.
\]
For a sequence of positive integers $\underline d=(d_1,\cdots, d_m)$, we set
\begin{align}
\label{Fd}
\mbb F(\underline d) = \mbb F(d_1)\otimes \cdots \otimes  \mbb F(d_m).
\end{align}
to be the tensor product of $\mathfrak s$-modules $\mbb F(d_1),\cdots, \mbb F(d_m)$. 

Let $H$ be the unital associative algebra over $\mbb C$ generated by $p_i$, $q_i$ for $i\in \mbb N$ and subject to the following relations:
\begin{align}
p_i p_j = p_j p_i, \quad
q_i q_j = q_j q_i,\quad
p_i q_j = q_j p_i + \delta_{i j} 1, \quad \forall i, j\in \mbb N.
\end{align}
This is the quotient algebra of the universal enveloping algebra of the  Heisenberg algebra $\mathfrak s$ by the two-sided ideal generated by $c=1$. 
We have that   the algebra $H$ is isomorphic to the algebra $H'$ generated by $a_i $ and $b_i$ for $i\in \mbb N$ and  subject to the relations 
\begin{align}
a_i a_j = a_j a_i, \quad
b_i b_j =b_j b_i,\quad
a_i b_j = \sum_{k=0}^{\min (i, j)}  b_{j-k}a_{i-k}, \quad \forall i, j\in \mbb N,
\end{align}
where we set $a_0=1=b_0$. 

It can be deduced as follows. 
Consider the positive self dual Hopf algebra $R$ with a unique irreducible primitive element  defined in ~\cite[Chapter I]{Z81}. As an associative algebra, $R$ is isomorphic to the Fock space $\mbb F$.
Let $\Delta: R\to R\otimes R$ be the comultiplication on $R$ defined by $\Delta (x_n ) =\sum_{i=0}^n x_{n-i} \otimes x_i$ for any $n=1, 2, \cdots$. Let $\langle-, -\rangle : R\times R \to \mbb Z$ be the associated bilinear form.  
For any element $x\in R$, we define a linear map $x^*: R\to R$ by 
\[
x^* (y) = \mbox{id}\otimes \langle x, -\rangle ( \Delta(y)), \quad \forall y\in R.  
\]
Let $x: R\to R$ denote the linear map by multiplying $x$. 
By ~\cite[1.9]{Z81}, we see that the assignment $a_i \mapsto x_i$ and $b_i \mapsto x_i^*$, for any $i=1, 2, \cdots$, define a faithful  $H'$-action on $\mbb F$, i.e., an injective algebra homomorphism $H' \hookrightarrow \End (\mbb F)$.

Let $\rho_i$ be the primitive element in $\mbb F_i$ defined by  
\[
\langle \rho_i, \sum_{\substack{ k+l =i \\k, l >0}} \mbb F_k \mbb F_l \rangle =0, \quad  \langle \rho_i, \rho_i \rangle =1,
\]
where $\mbb F_i$ is the homogeneous component of $\mbb F$ of degree $i$.
Again by ~\cite[1.9]{Z81}, the assignment $p_i\mapsto \rho_i$ and $q_i\mapsto \rho_i^*$, for any $i=1, 2, \cdots$, define a faithful $H$-action on $\mbb F$: $H\hookrightarrow \End(\mbb F)$.

The collections $\{ x_i |i=1, 2, \cdots,\} $ and $\{ \rho_i| i=1, 2, \cdots\}$ both generate the algebra $\mbb F$. So the images of the subalgebras generated by
$p_i$ and $a_i$ for any $i=1, 2, \cdots$ in $H$ and $H'$ respectively coincide in $\End (\mbb F)$. Now the fact that  $(x+y)^* = x^* + y^*$ implies that the images of
$H$ and $H'$ coincide in $\End(\mbb F)$.  This shows that the two algebras are isomorphic to each other by taking into consideration of the gradings.

From the above analysis, we see that in order to define an action of $\mathfrak s$ on $\mbb F$ isomorphic to $\mbb F(1)$, 
it is equivalent to define an $H'$-action on the Fock space $\mbb F$. 
But an $H'$-action on $\mbb F$ is completely determined if the Fock space $F$ is associated with an algebra structure isomorphic to  the algebra $R$.

\section{Main result}

\subsection{An analogue of Springer resolution}
\label{Springer-1}

Let $\Gamma$ be the framed Jordan graph
 
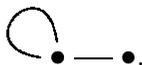
\[
\xy 
(30, 30) *+{\bullet}; (30, 30)*+{\bullet}; 
**\crv{(15, 30) & (30, 45) }; 
(40, 30)*+{\bullet.}
**\crv{(35, 30)}
\endxy
 \]
We fix an orientation  $Q$ and its opposite $\bar Q$ for $\Gamma$:  
\[
Q:  \xymatrix{
\bullet  \ar@(ul, dl) []_{\sigma}  \ar @{->}[r] ^{\rho} & \bullet 
},
\quad \quad
\bar Q: 
\xymatrix{
\bullet  \ar@(dl, ul) []^{\bar \sigma}  \ar @{<-}[r] ^{\bar \rho} & \bullet .
}
\]

To a pair $(\nu,d)$ of non negative integers, we fix a pair $(V, D)$ of vector spaces of dimensions $\nu$ and $d$, respectively.  We define  the representation space of $Q$ of dimension vector $(\nu, d)$ to be 
\[
\E_{\nu, d, Q } = \End (V) \times \Hom (V, D).
\]
This definition depends on the choice of the pair $(V, D)$.  However different choices give rise to isomorphic spaces. So we may neglect this ambiguity. If we want to emphasis the choice of the pair of vector spaces, we write $\E_{V, D, Q}$ instead of 
 $\E_{\nu, d, Q}$.
We denote by $x = (x_{\sigma}, x_{\rho})$ the elements in $\E_{\nu, d, Q}$ with $x_{\sigma}$ and $x_{\rho}$ in $\End(V)$ and $\Hom(V, D)$, respectively. 

Let us fix a composition $\underline d=(d_1,\cdots, d_m)$ of $d$, i.e., $d_1+\cdots+ d_m=d$. Without lost of generalities, we further assume that $d_i\neq 0$ for $1\leq i \leq m-1$.  We also fix a flag $\underline D$ of type $\underline d$:
\begin{align}
\label{fixedD}
\underline D = (D= D_1\supseteq \cdots \supseteq D_m\supseteq D_{m+1}=0),
\end{align}
such that $\dim D_i/D_{i+1}= d_{i}$ for any $i=1, \cdots, m$.

We call $\underrightarrow{\nu}= (\underline{ \nu_1},\cdots, \underline {\nu_m})$ a $multi$-$composition$ of $\nu$ if 
$\underline{ \nu_i}$ is a composition of $\nu_i$ for $1\leq i\leq m$ and 
$(\nu_1,\cdots, \nu_m)$ is a composition of $\nu$. Note that  $\underrightarrow{\nu}$ can be regarded as a composition of $\nu$ as well.
Given any flag $\underline V$ of type $\underrightarrow{\nu}$, we set $V_i$ to be the position at the flag such that 
$\dim V_i/V_{i+1} = \nu_i$ for any $1\leq i \leq m$.
We also set $V_{i, j}$ to  be the position such that $\dim V_{i, j}/V_{i, j+1} = \nu_{i, j}$, i.e.,  the $j$-th position in the composition $\underline{\nu_i}$.
In particular, we have $V_i = V_{i, 1}$ for any $1\leq i\leq m$.
For convenience, we set
$V_{i, a_i+1}=V_{i+1, 1}$ if $\underline{\nu_i}=(\nu_{i, 1},\cdots, \nu_{i, a_i})$.

Let $\F_{\underset{\rightarrow}{\nu}}$ be the variety of all flags of type $\underrightarrow{\nu}$ in $V$.

A pair $(x, \underline V)$ in $\E_{\nu, d, Q}\times \Fv$ is called $a$ $stable$ $pair$  if 
\begin{itemize}
\item $x_{\sigma}(V_{i, j}) \subseteq V_{i, j+1}$, for any $1\leq i\leq m$ and $1\leq j\leq a_i$;
\item $x_{\rho} (V_i) \subseteq D_i$, for any $1\leq i\leq m$.
\end{itemize}

Let $\tFv$ be the variety of all stable pairs in $\E_{\nu, d, Q}\times \Fv$. We have the following diagram
\begin{align}
\label{Springer}
\begin{CD}
\Fv @<\tau <<
\tFv @>\pi_{\underset{\to}{\nu}} >> \E_{\nu, d, Q} ,
\end{CD}
\end{align}
where the first map is the second projection and the second map is the first projection. 
Since $\Fv$ is a projective variety,  $\pi_{\underset{\to}{\nu}}$ is projective, hence proper. 

Let $\tNv$ be the variety of all stable pairs $(x_{\sigma}, \underline V)$, (where  $x_{\rho}$ is regarded as $0$). Then $\tau$ factors through $\tNv$:
\begin{align}
\label{factor}
\begin{CD}
\tFv  @>\tau_1 >> \tNv @>\tau_2 >> \Fv,
\end{CD}
\end{align}
where $\tau_1$ forgets $x_{\rho}$ and $\tau_2$ forgets $x_{\sigma}$.  Observe that $\tau_1$ and $\tau_2$ are vector bundles of respective fiber dimensions:
\[
f_1= \sum_{i\leq i'} \nu_{i} d_{i'},
\quad
f_2= \sum_{(i, j)< (i', j') }  \nu_{i, j} \nu_{i', j'},
\]
where $(i, j) < (i', j')$ is the lexicographic order.
So, $\tau$ is a vector bundle of fiber dimension $f_1+f_2$.
We see that $\tFv$ is a smooth irreducible variety of dimension $f_1+f_2 +\dim \Fv$, which is 
\begin{align}
\label{dimension-E}
 f=\sum_{i\leq i'} \nu_{i} d_{i'} + 2 \sum_{(i, j)< (i', j') }  \nu_{i, j} \nu_{i', j'}.
\end{align}

\begin{prop}
\label{pi-small}
The morphism $\pi_{\underset{\to}{\nu}}: \tFv \to \E_{ \underset{\to}{\nu} }$, where $\E_{ \underset{\to}{\nu} } =\mrm{im} (\pi_{\underset{\to}{\nu}})$, is semismall and a resolution of singularities.
\end{prop}

\begin{proof}
From (\ref{global}), it is reduced to show that 
\[
\dim \tFv \times_{\E_{\nu, d, Q}} \tFv \leq \dim \tFv.
\]
Let 
\begin{align}
\label{mu}
\mu_{\underset{\to}{\nu}}: \tNv \to \End(V)
\end{align}
 be the first projection, which is the generalized Springer resolution. 
Let $\Nv$ be the image of $\mu_{\underset{\to}{\nu}}$. 
Let $\mathcal O_{\lambda}$, where $\lambda=(\lambda_1\geq \lambda_2 \geq\cdots \geq \lambda_m)$,  be the nilpotent $\G$-orbit in $\End(V)$ such that the sizes of  the Jordan blocks in its Jordan canonical form  are $\lambda_1, \lambda_2 , \cdots , \lambda_m$.  
Then the variety  $\Nv$ admits a stratification
$\Nv = \sqcup \mathcal O_{\lambda}$ where the union runs over all $\lambda$ such that $\mathcal O_{\lambda}$ is in $\Nv$.
It is well known that the map $\mu_{\underset{\to}{\nu}}$ is strictly semismall with respect to this stratification. 
Now, this stratification induces a stratification $\sqcup \tNv(\lambda)$ on $\tNv$ and a stratification $\sqcup \tFv(\lambda)$ of $\tFv$, where $\tFv(\lambda) = \tau_1^{-1} \mu_{\underset{\to}{\nu}}^{-1}(\mathcal O_{\lambda})$, 
since $\tau_1$ in (\ref{factor}) is a vector bundle.  
So we have a stratification 
\[
\tFv \times_{\E_{\nu, d, Q}} \tFv  = \sqcup \tFv (\lambda) \times_{\E_{\nu, d, Q}} \tFv  (\lambda).
\]
It is clear that $\pi_{\underset{\to}{\nu}}^{-1}(x_{\sigma}, x_{\rho}) \subseteq \mu_{\underset{\to}{\nu}}^{-1} (x_{\sigma})$. So we have
\begin{align*}
\dim  & \tFv \times_{\E_{\nu, d,Q}} \tFv   =\mbox{sup}_{\lambda} \dim   \tFv (\lambda) \times_{\E_{\nu, d, Q}} \tFv  (\lambda)\\
&\leq \mbox{sup}_{\lambda} \{ \dim \tFv(\lambda) + \dim  \pi_{\underset{\to}{\nu}}^{-1} (x_{\sigma}, x_{\rho})\}
\leq  \mbox{sup}_{\lambda} \{ \dim \tFv(\lambda) + \dim  \mu_{\underset{\to}{\nu}}^{-1} (x_{\sigma})\},
\end{align*}
where $(x_{\sigma}, x_{\rho})\in \mathcal O_{\lambda}\times \Hom (V, D)$. Since $\mu_{\underset{\to}{\nu}}$ is strictly semismall, we have 
\[
2 \dim   \mu_{\underset{\to}{\nu}}^{-1} (x_{\sigma}) = \dim \tNv - \dim \mathcal O_{\lambda}, \quad  \forall x_{\sigma}\in \mathcal O_{\lambda},
\]
which can be rewritten as 
\[
\dim   \mu_{\underset{\to}{\nu}}^{-1} (x_{\sigma}) = \dim \tNv - \dim \tNv(\lambda) =\dim \tFv-\dim \tFv (\lambda),
\]
for any $x_{\sigma}\in \mathcal O_{\lambda}$, where the last equality is due to the fact that $\tau_1$ is a vector bundle.
So we have 
$\dim \tFv(\lambda) + \dim  \mu_{\underset{\to}{\nu}}^{-1} (x_{\sigma}) = \dim \tFv$.
Hence we have 
$\dim \tFv \times_{\E_{\nu, d,Q}} \tFv \leq \dim \tFv$. This shows that $\pi_{\underset{\to}{\nu}}$ is semismall.

Let $\mathcal O_{\lambda} \subseteq \Nv$ be the open dense orbit in $\Nv$. 
Let $X_0 = \E_{\underset{\to}{\nu}} \cap \mathcal O_{\lambda}\times \Hom(V, D)$. 
Since $(x_{\sigma}, 0)\in \E_{\underset{\to}{\nu}}$ for any $x_{\sigma} \in \mathcal O_{\lambda}$, we see that 
$X_0$ is non empty. So $X_0$ is an open dense subvariety in $\E_{\underset{\to}{\nu}}$. 
Moreover, the restriction $\pi_{\underset{\to}{\nu}}^{-1}( X_0) \to X_0$ is an isomorphism due to the fact that 
the restriction  map $\mu_{\underset{\to}{\nu}}^{-1} (\mathcal O_{\lambda}) \to \mathcal O_{\lambda}$ is an isomorphism.  
Therefore, $\pi_{\underset{\to}{\nu}}$ is a resolution of singularity.  The proposition follows.
\end{proof}

Note that $\E_{\underset{\to}{\nu}}$ is irreducible due to the fact that $\tFv$ is irreducible. It is also clear that if $x\in \E_{\underset{\to}{\nu}}$, then $x_{\sigma}$ is nilpotent. 
So $\E_{\underset{\to}{\nu}}$ is contained in $\Nv\times \Hom (V, D)$.

In what follows, we shall produce a second proof of Proposition ~\ref{pi-small}.  We set
\[
Y = \tFv \times_{\E_{\nu, d, Q}} \tFv , \quad Z=\tNv \times_{\Nv} \tNv.
\]
The factorization (\ref{factor}) of the map $\pi_{\underset{\to}{\nu}}$ induces the following maps
\begin{align}
\label{factor-2}
Y \to Z \to \Fv\times \Fv.
\end{align}

The diagonal action of $\G$ on $\Fv\times \Fv$ has only finitely many orbits parametrized by the set 
$\Theta$ of all square matrices $M=(m_{(i, j), (k, l)})_{1\leq i, k \leq m, 1\leq j \leq a_i, 1\leq l \leq a_k}$ of  size $\nu=\nu_1+\cdots +\nu_n$ such that
\begin{itemize}
\item $m_{(i, j), (k, l)} \in \mbb N$;
\item $\sum_{ k, l} m_{(i, j), (k, l)} = \nu_{i, j} $ and  $\sum_{i, j} m_{(i, j), (k, l)} = \nu_{k, l}$. 
\end{itemize} 
The correspondence of the set of $\G$ orbits and the set $\Theta$ is given by assigning  a pair $(\underline V, \underline V')$ of flags to the matrix $M$ whose entry $m_{(i, j), (k, l)}$ is defined by
\[
\dim \frac{V_{i, j+1}  + V_{i, j } \cap V'_{k, l}} { V_{i, j+1} + V_{i, j } \cap V'_{k, l+1}},
\]
where, by convention, $V_{i, 0} = V_{i-1, a_{i-1}}$ and $V_{i, a_i +1} = V_{i+1, 1}$. 

Let $\mathcal O_M$ be the $\G$ orbit indexed by $M$. So we have a stratification $\sqcup \mathcal O_M$ of $\Fv\times \Fv$.  Such a stratification induces stratifications on $Y$ and $Z$:
\[
Y = \sqcup Y_M, \quad Z=\sqcup Z_M, 
\]
where $Y_M$ and $Z_M$ are the preimage of the orbit $\mathcal O_M$ under the morphisms in (\ref{factor-2}). 
From ~\cite[2.1]{BLM90}, we have
\[
\dim \mathcal O_M = \sum m_{(i, j), (k, l)} m_{(i', j'), (k', l')},
\]
where the sum is over the quadruples such that $(i, j)> (i', j') $ or $(k,l) > (k', l')$. 
Moreover the map $Z_M\to O_M $ is a vector bundle of fiber dimension 
$\sum m_{(i, j), (k, l)} m_{(i', j'), (k', l')}$ where the sum is over all quadruples such that $(i, j) < (i', j')$ and $(k, l) <(k', l')$.
We have that $Z_M$ is a locally closed, smooth,  and irreducible subvariety of $Z$ of dimension 
\[
\dim Z_M = \nu^2 - \sum_{(i, j)} \nu_{i,j}^2.
\]
Similarly, the map $Y_M \to Z_M$ is a vector bundle of fiber dimension $ \sum_{i, k \leq r} n_{i, k} d_r$, where 
\[
n_{i,k} =\dim \frac{V_{i+1} + V_i \cap V'_k}{V_{i+1}+ V_i \cap V'_{k+1}}, \quad \forall  (\underline V, \underline V')\in \mathcal O_M.
\]
So the variety $Y_M$ is a locally closed, irreducible, subvariety of $Y$ of dimension 
\[
\dim Y_M =  \nu^2 - \sum_{i, j} \nu_{i,j}^2 + \sum_{i, k \leq r} n_{i, k} d_r.
\]
Observe that 
\begin{align}
\label{estimate}
\sum_{i, k \leq r} n_{i, k} d_r \leq \sum_{i \leq r} \nu_i d_r.
\end{align}
So we have 
\[
\dim Y_M \leq  \nu^2 - \sum_{i, j} \nu_{i,j}^2 + \sum_{i \leq r} \nu_i d_r = \dim \tFv,
\]
for any $M\in \Theta$. We have proved in a second way  that the morphism $\pi_{\underset{\to}{\nu}}$ is semismall.

Moreover, we have

\begin{cor}
\label{relevant}
$\dim Y_M =\dim \tFv$ if and only if $M$ is a diagonal block matrix, whose diagonal blocks have the sizes   $\nu_1, \cdots, \nu_m$.
\end{cor}

Another way to state Corollary ~\ref{relevant} is as follows.

\begin{cor}
\label{irreducible}
The irreducible components in $Y$ of largest dimension equal to $\dim Y$ are of the form $\overline{Y_M}$ 
where $M$ is a diagonal block matrix, whose diagonal blocks have the sizes   $\nu_1, \cdots, \nu_m$.
\end{cor}

In particular, we have

\begin{cor}
The morphism $\pi_{\underset{\to}{\nu}}$ is small if each partition $\underline \nu_i$ in $\underset{\to}{\nu}$ consists of only one part for any $1\leq i\leq m$.
\end{cor}

\subsection{A class  of simple perverse sheaves}

We set
\[
L_{\underset{\to}{\nu}} = (\pi_{\underset{\to}{\nu}})_! (\mbb C_{\tFv})[\dim \tFv].
\]
Since $\tFv$ is smooth and $\pi_{\underset{\to}{\nu}}$ is proper, we see that $L_{\underset{\to}{\nu}}$ is semisimple by the Decomposition theorem (see ~\cite{BBD82}). 
By Proposition ~\ref{pi-small}, we see that $L_{\underset{\to}{\nu}} $ is  a perverse sheaf. We refer to ~\cite{CG97} and the references therein for a proof of this fact. 
So $L_{\underset{\to}{\nu}} $ is a direct sum of simple perverse sheaves (without shifts) on $\E_{\nu, d, Q}$, i.e.,
\[
L_{\underset{\to}{\nu}}  = \oplus P\otimes V_P,
\]
where $P$ runs over all simple perverse sheaves on $\E_{\nu, d,Q}$ and $V_P$ is certain vector subspace determined by $P$ and $L_{\underset{\to}{\nu}} $.
Furthermore, since $\pi_{\underset{\to}{\nu}}$ is a resolution of singularities, we have 
\begin{align}
\label{leading}
L_{\underset{\to}{\nu}} = \IC(\E_{\underset{\to}{\nu}}) \oplus M,
\end{align}
where $M$ is a perverse sheaf such that its support is strictly contained in $\E_{\underset{\to}{\nu}}$.  If $\underset{\to}{\nu}$ is a composition of partitions $\underline \lambda= (\lambda_1,\cdots, \lambda_m)$, we write
\[
\IC (\E_{\underline \lambda}) = \IC(\E_{\underset{\to}{\nu}}).
\]

\begin{lem}
\label{unique}
We have $\IC (\E_{\underline \lambda}) \neq \IC (\E_{\underline{ \lambda'}})$ if $\underline \lambda \neq \underline {\lambda'}$, i.e., $\lambda_i \neq \lambda_i'$ for some $i$.
\end{lem}

\begin{proof}
When $m=1$, the statement  is clear because the variety $\E_{\underline \lambda}$ is  the closure of the nilpotent orbit in $\End(V)$ whose Jordan type is $\lambda_1^{\perp}$, the dual partition of $\lambda_1$.

We now prove the statement for $m=2$, i.e., $\underline \lambda = (\lambda_1, \lambda_2)$. 
It suffices to show that $\E_{\underline \lambda} \neq \E_{\underline {\lambda'}}$.
This is true if the associated partitions of $\underline \lambda$ and $\underline {\lambda'}$ are different because  the projections of the two varieties to $\End (V)$ are different.
Suppose that $\underline \lambda$ and $\underline {\lambda'}$ have the same associated partition, say $\mu$. 
Let $\mu^{\perp}$ be its dual partition. 
Then $X_{0, \underline \lambda} =\E_{\underline \lambda} \cap (\mathcal O_{\mu^{\perp}} \times \Hom(V, D) )$ is  open dense in $\E_{\underline \lambda}$. 
Thus, to show that $\E_{\underline \lambda} \neq \E_{\underline {\lambda'}}$, it is enough to show that $X_{0, \underline \lambda} \neq X_{0, \underline{ \lambda'}}$.  

Fix an element $x_{\sigma} \in \mathcal O_{\mu^{\perp}}$, there exists a unique flag $\underline V$ (resp. $\underline V'$) of type $\underline \lambda$ (resp. $\underline{\lambda'}$) such that 
the pair $(x_{\sigma}, \underline V)$ (resp. $(x_{\sigma}, \underline V')$)  is a stable pair.  (The uniqueness of the flags $\underline V$ and  $\underline V'$  is due to the fact that the map $\mu_{\underset{\to}{\nu}}$ in (\ref{mu}) is  a resolution of singularities.)
Recall that $V_2$ is the step at $\underline V$ such that $\dim V_2 = |\lambda_2|$. Let $x_{\sigma}|_{V_2}$  be the restriction of $x_{\sigma}$ to $V_2$. 
Then the type of   $ x_{\sigma}|_{V_2}$  is  $\lambda_2^{\perp}$ due to the fact that $\underline V$ is unique. 
Similarly, the type of $x_{\sigma}|_{V'_2}$ is $(\lambda'_2)^{\perp}$.
The assumption that  $\underline \lambda \neq \underline {\lambda'}$ implies  $\lambda_2\neq \lambda_2'$. So the restrictions   $x_{\sigma}|_{V_2} $ and $x_{\sigma}|_{V_2'}$ have different types $\lambda_2^{\perp}$ and $(\lambda_2')^{\perp}$, respectively. 
Thus $V_2\neq V_2'$. 

By definition, we know that a pair $(x_{\sigma}, x_{\rho})$ is in $X_{0,\underline \lambda}$ (resp. $X_{0, \underline{\lambda'}}$) if and only if $x_{\rho} (V_2)\subseteq D_2$
(resp. $x_{\rho}(V'_2)\subseteq D_2$). Since $V_2\neq V_2'$, there exists an element $x_{\rho}$  such that $x_{\rho}(V_2)\subseteq D_2  $ and $x_{\rho}(V_2')\not \subseteq D_2$.
This implies that  $(x_{\sigma}, x_{\rho})\in X_{0, \underline \lambda} -X_{0, \underline{ \lambda'}}$.
Thus, $X_{0, \underline \lambda} \neq X_{0, \underline{ \lambda'}}$. The statement for $m=2$ holds, then.

The general case can be proved in a similar way. We leave it to the reader.
\end{proof}

\begin{lem}
\label{nu-lambda}
We have 
$\E_{\underset{\to}{\nu}} = \E_{\underline \lambda}$ if $\underline{\nu_i}$ is obtained from $\lambda_i$ by a permutation of the parts in  $\lambda_i $ for any $1\leq i\leq n$.
\end{lem}

\begin{proof}
Let $\xi_i$ be the dual partition associated to $ \lambda_i$ and $\underline{\nu_i}$ for any $1\leq i\leq m$. 
For any pair $(x, \underline V) $ in $\E_{\underline \lambda}$, let $x_i $ be the subquotient of $x$ to $V_i/V_{i+1}$ and 
$\underline {W_i}$ be the flag in $V_i/V_{i+1}$ of type $\lambda_i$ obtained from $\underline V$ for any $1\leq i\leq m$. 
Then the pair $(x_i, \underline{W_i})$ is a stable pair for any $1\leq i\leq m$. So $x_i \in \mathcal N_{\lambda_i} = \mathcal N_{\underline{\nu_i}}$. 
Thus, there exists a flag, say $\underline{W'_i}$ of type $\underline{\nu_i}$ such that $(x_i, \underline{W_i'})$ is a stable pair for any $1\leq i\leq m$.
The collection $\{\underline {W_i'}|1\leq i\leq m\}$ determines a unique flag $\underline {V'}$  in $V$ of type $\underset{\to}{\nu}$ and it is clear that $(x, \underline{V'})$ is a stable pair. So $x\in \E_{\underset{\to}{\nu}}$.
Hence, $\E_{\underline \lambda}\subseteq \E_{\underset{\to}{\nu}}$. Similarly we have $\E_{\underset{\to}{\nu}}\subseteq \E_{\underline \lambda}$.
\end{proof}

Let  $\underline 1$ be the  composition of partitions such that the $i$-th part is the partition $(1, \cdots, 1)$ of $\nu_i$ consisting of all   $1$'s.  We denote by 
\[
\underline 1\leq \underline \lambda
\]
if  $|\lambda_i|=\nu_i$, the number of parts in the $i$-th partition of $\underline 1$  for any $1\leq i\leq m$. In particular, the $i$-th partition of $\underline 1$ is less than the partition $\lambda_i$ 
for any $1\leq i \leq m$.  
It is clear that 
\[
\cup_{\underline \lambda: \underline 1\lneqq \lambda} \E_{\underline \lambda} \subseteq  \E_{\underline 1} \backslash \E_{\underline 1}^{ (\nu)}.
\]
But this inclusion may be a strict inclusion.  We set
\[
\E_{\underline 1}^{\xi} = \E_{\underline 1} \cap ( \mathcal O_{\xi}\times \Hom(V, D)).
\]
In particular, $\E_{\underline 1}^{(\nu)}$ contains all elements in $\E_{\underline 1}$ such that $x_{\sigma}$ is regular nilpotent.

\begin{prop}
\label{relevant-1}
Let $\E_{\underline 1} = \sqcup X_i$ be a stratification such that 
$X_i \subseteq \E_{\underline 1}^{\xi}$ and the restriction map $\pi_{\underline 1}^{-1}(X_i) \to X_i$ is a locally trivial fibration.  
If  $X_i\subseteq \E_{\underline 1}\backslash \E_{\underline 1}^{(\nu)} $ is a relevant stratum, then 
$X_i \subseteq   \E_{\underline \lambda}$ for some $\underline \lambda$ such that $\underline 1\lneqq \underline \lambda$. 
\end{prop}

\begin{proof}
Let $\tilde X_i =\pi_{\underline 1}^{-1}(X_i)$. The assumption is equivalent to say that $Y_i =\tilde X_i \times_{X_i} \tilde X_i $ 
has the same dimension as that of $ \widetilde{\mathcal F}_{\underline 1}$ (see (\ref{Springer})).
Since the restriction of $\pi_{\underline 1}$ is a fibration,  we have a stratification of $\tilde X_i =\sqcup \tilde X_i^{\alpha}$ 
where the $\tilde X_i^{\alpha}$'s  are the irreducible components of $\tilde X_i$.  
This induces a stratification of $Y_i =\sqcup Y_i^{\alpha, \beta}$ where $Y_i^{\alpha, \beta} = \tilde X_i^{\alpha} \times_{X_i} \tilde X_i^{\beta}$ are the irreducible components. Thus the dimension 
of $Y_i^{\alpha, \beta}$ is equal to the dimension of $\widetilde{\mathcal F}_{\underline 1}$.  
By Corollary ~\ref{relevant}, we see that 
$Y_i^{\alpha, \beta} \cap Y_M$, for any $\alpha, \beta$,  is open dense in $Y_i^{\alpha, \beta}$ for some diagonal block matrix $M \neq 1$, 
whose sizes of the diagonal blocks from the top to the bottom are
$\nu_1, \cdots, \nu_m$. 

If $M$ is such a block matrix, then we have  $V_j=V_j'$ for any $1\leq j\leq m$,  for any element  $(x, \underline V, \underline V') \in Y_i^{\alpha, \beta} \cap Y_M$. 
Since $M\neq 1$, there is a diagonal block, say $M_k$, not equal to $1$.  
This implies that the subquotient $x_{\sigma}|_{V_k/V_{k-1}}$ fixes two different complete flags of $V_k/V_{k+1}$, say $\underline W$ and $\underline W'$,   
obtained from  $\underline V$ and $\underline V'$ respectively.
So  the type of $x_{\sigma}|_{V_k/V_{k+1}}$ is not $\nu_k =\dim V_k/V_{k+1}$, as a partition of $\nu_k$. In other words,   $x_{\sigma}|_{V_k/V_{k+1}}$ is not regular nilpotent. 
This induces that there is a position, say $l$,  in the complete flag $\underline W$ such that  $x_{\sigma}|_{V_k/V_{k+1}} (W_l) \subseteq W_{l+2}$.
Hence,  the pair  $(x_{\sigma}|_{V_k/V_{k+1}}, \check{\underline W})$ is a stable pair,
where  $\check{\underline W}$ is  the partial flag obtained from $\underline W$ by deleting $W_l$. 
By Lemma ~\ref{nu-lambda}, this implies that $x\in \cup_{\underline \lambda: \underline 1\lneqq \lambda} \E_{ \underline \lambda}$. 
Subsequently, $\pi^2_{\underline 1}(\cup_{ M} Y_i \cap Y_M) \subseteq \cup_{\underline \lambda: \underline 1\lneqq \lambda} \E_{ \underline \lambda}$, where 
$\pi_{\underline 1}^2$ is the obvious projection and $M$ runs over all diagonal block matrix $M$ of diagonal block size $\nu_1, \cdots, \nu_m$. 

Let $X_i^c$ be the  open subset of all  elements in $X_i$ such that its fiber $(\pi_{\underline 1}^2)^{-1} (X_i^c)$ is disjoint from $\cup_{M} Y_i \cap Y_M$ 
for any  diagonal block matrix $M $ of diagonal block sizes $\nu_1, \cdots, \nu_m$. 
Since $Y_i^{\alpha, \beta} \cap Y_M$ is open dense in $Y_i^{\alpha, \beta}$, this implies that the dimension of $(\pi_{\underline 1}^2)^{-1} (X_i^c)$ is strictly less than the dimension of $Y_i$.
This   forces  $X_i^c=\mbox{\O}$ because $X_i$ is irreducible.  
Therefore,  $X_i  \subseteq \cup_{\underline \lambda: \underline 1\lneqq \lambda} \E_{ \underline \lambda}$. 
Since $X_i$ is irreducible and $\E_{\underline \lambda}$ is closed, we have  $X_i \subseteq  \E_{\underline \lambda}$ for some $\underline \lambda$. 
\end{proof}

\begin{prop}
\label{relevant-2}
Suppose that $X_i \subseteq \E_{\underline 1}^{\xi}$ is a relevant stratum. Then there exists a composition $\underline \lambda$ of partitions such that  $\underline 1\leq \underline \lambda$ and
the associated dual partition of $\underline \lambda$ is $\xi$ and $\overline{X_i} =\E_{\underline \lambda}$ where $\overline{X_i}$ is 
the closure of $X_i$.
\end{prop}

\begin{proof}
For two compositions $\underline \lambda$ and $\underline {\lambda'}$ of partitions such that $\underline 1\leq \underline \lambda$ and $\underline 1\leq \underline {\lambda'}$, we denote by 
$\underline \lambda \leq \underline {\lambda'}$ if $\lambda_i \leq \lambda_i'$ for any $1\leq i\leq m$.
By Proposition ~\ref{relevant-1}, we can choose a composition $\underline \lambda$ of partitions
such that  $X_i \subseteq \E_{\underline \lambda}$ and 
$\underline \lambda$ is a maximal one among those compositions of partitions greater than  $\underline 1$.

Assume that the associated dual partition of  $\underline \lambda$  is not $\xi$.
Let $X_i^0$ be the open dense stratum in a stratification of $X_i$ 
such that the restriction of $\pi_{\underline \lambda}$ to $\pi_{\underline \lambda}^{-1} ( X_i^0) \to X_i^0$ is a locally trivial fibration.
If $X_i^0$ is a relevant stratum, then we can show in a similar way as the proof of Proposition ~\ref{relevant-1} that 
$X_i^0 \subseteq \cup_{\underline \lambda \lneqq \underline{\lambda'}} \E_{\underline {\lambda'}}$. 
This contradicts with the choice of $\underline \lambda$. Thus, we see that $X_i^0$ is an irrelevant stratum with respect to $\pi_{\underline \lambda}$.
Hence we have 
\[
2\dim \pi_{\underline \lambda}^{-1}(x) < \dim \E_{\underline \lambda} - \dim X_i^0 =  \dim \E_{\underline \lambda} - \dim X_i, \quad \forall x\in X_i^0.
 \]
 Let $\widetilde X_i(\underline \lambda)$ be the subvariety in $\widetilde X_i :=\pi_{\underline 1}^{-1}(X_i)$ consisting of all stable pairs $(x, \underline V)$ such that 
 the associated flag of type $\underline \lambda$ together with  $x$ is a stable pair. 
 Since  $\underline \lambda$ is the maximal one, we see that $\dim \widetilde X_i(\underline \lambda) =\dim \widetilde X_i$.
 Let $p_{\underline \lambda}: \widetilde X_i(\underline \lambda) \to \pi_{\underline \lambda}^{-1}(X_i)$ be the projection determined by sending the complete flag to the associated flag of type $\underline \lambda$. 
 Then  the fiber dimension of  $p_{\underline \lambda}$ is $\frac{1}{2}\sum_{i=1}^m \nu_i^2-\nu_i$. 
 By combining the above results, we have 
 \begin{align*}
 2\dim \pi_{\underline 1}^{-1} (x)  = 2\dim (p_{\underline \lambda} \pi_{\underline \lambda})^{-1} (x) 
 &\leq 2 \dim p_{\underline \lambda}^{-1} (y) 
 + 2 \dim \pi_{\underline \lambda}^{-1} (x)  \\
& < \sum_{i=1}^m \nu_i^2-\nu_i +  \dim \E_{\underline \lambda} - \dim X_i, 
 \end{align*}
for any $x\in X_i^0$ and $\pi_{\underline \lambda}(y) =x$.  
By the formula (\ref{dimension-E}), $\sum_{i=1}^m \nu_i^2-\nu_i +  \dim \E_{\underline \lambda}=\dim \E_{\underline 1}$. 
By the fact that $\widetilde X_i\to X_i$ is a fibration,  we have 
\[
2\dim \pi_{\underline 1}^{-1} (x) < \dim \E_{\underline 1} -\dim X_i, \quad \forall  x\in X_i.
\] 
This contradicts with our assumption that $X_i$ is a relevant stratum with respect to $\pi_{\underline 1}$. So the dual partition of $\underline \lambda$ is $\xi$.
In this case, we have $\dim X_i =\dim \E_{\underline \lambda}$. 
Otherwise, $2\dim \pi_{\underline 1}^{-1}(x) =\dim \E_{\underline 1} -\dim \E_{\underline \lambda} < \dim \E_{\underline 1} -\dim X_i$ for any $x\in X_i$. A contradiction. Therefore,
 $\overline{X_i} =\E_{\underline \lambda}$. 
\end{proof}

The following corollary follows from  Propositions ~\ref{relevant-1} and ~\ref{relevant-2}.

\begin{cor} 
\label{L-coarse}
When $\underset{\to}{\nu} =\underline 1$,  the complex $L_{\underline 1} $ in (\ref{leading}) has the following decomposition. 
\[
L_{\underline 1} = \IC(\E_{\underline 1}) \oplus \oplus_{\underline \lambda: \underline 1\lneqq \underline \lambda}  \oplus_{\chi} \IC(\E_{\underline \lambda}, \mathcal L_{\underline\lambda, \chi}) \otimes V_{\underline \lambda, \chi},
\]
where $\mathcal L_{\underline \lambda, \chi}$ runs over the set of  irreducible local systems on an open dense smooth subvariety in $\E_{\underline \lambda}$ and $V_{\underline \lambda, \chi}$ is a certain 
finite dimensional vector space over $\mbb C$.
\end{cor}

\begin{thm}
\label{L-fine}
The  decomposition in Corollary \ref{L-coarse} can be refined as follows.
\[
L_{\underline 1} =  \oplus_{\underline \lambda: \underline 1\leq \underline \lambda} \IC(\E_{\underline \lambda}) \otimes V_{\underline \lambda},
\]
where $V_{\underline \lambda}$ is the simple module of $S_{\nu_1}\times \cdots \times S_{\nu_m}$ parametrized by $\underline \lambda$.
\end{thm}

The proof of this theorem will be given in Section ~\ref{Grothendieck}.

Let $\mathcal S_{\nu, \underline d}$ be the set of all isomorphism classes of simple perverse sheaves in $L_{\underset{\to}{\nu}}$ for various $\underset{\to}{\nu}$.  
By abuse of notation, we denote by $\IC (\E_{\underline \lambda}) $ its isomorphism class. 
Let $p(n)$ be the number of partitions of $n$, where $n$ is a positive integer.  We set $p(0)=1$.

\begin{cor}
\label{dimension-S}
$\mathcal S_{\nu, \underline d} =\{ \IC (\E_{\underline \lambda})  \}_{\underline \lambda}$, 
where $\underline \lambda$ runs over all compositions $\underline \lambda = (\lambda_1,\cdots, \lambda_m)$ of partitions such that $|\lambda_1| +\cdots +|\lambda_m|=\nu$. 
In particular,   We have 
\begin{align}
\label{dimension}
\# \mathcal S_{\nu, \underline d} = \sum_{(\nu_1,\cdots, \nu_m)} p(\nu_1) p(\nu_2) \cdots p(\nu_m),
\end{align}
where the sum runs over all tuples $(\nu_1,\cdots, \nu_m)$ of nonnegative integers such that $ \nu_1+\cdots  + \nu_m =\nu$. 
\end{cor}

\begin{proof}
The is because all simple perverse sheaves in $L_{\underset{\to}{\nu}}$ are in $L_{\underline 1}$ for certain $\underline 1$ and all simple perverse sheaves in $L_{\underline 1}$ are of the form $\IC(\E_{\underline \lambda})$. 
\end{proof}

In the case  when $\underset{\to}{\nu} = \underline 1$, we set $Y_{\underline 1} $  for  $Y$ in (\ref{factor-2}). 
In this case, the $\G$-orbits in $\F_{\underline 1}\times \F_{\underline 1}$ are parametrized by the permutation group $S_{\nu}$. 
We regard the group $S_{\nu_1} \times  \cdots \times S_{\nu_m}$ as a subgroup of $S_{\nu}$ via the natural imbedding.
Then the top Borel-Moore homology 
\[
H_{top} (Y_{\underline 1}) = H_{2\dim Y_{\underline 1}} (Y_{\underline 1})
\]
of $Y_{\underline 1}$
has dimension $\# S_{\nu_1} \times \cdots \times S_{\nu_m}$ because 
\[
H_{top} (Y_{\underline 1}) =\mbox{span} \{ [Y_M]  | M \in S_{\nu_1} \times \cdots \times S_{\nu_m}\},
\]
by Corollary ~\ref{irreducible} and ~\cite[2.6]{CG97}.
Moreover
\[
H_{top} (Y_{\underline 1}) \simeq \End (L_{\underline 1}),
\]
by ~\cite[8.9.8]{CG97}, where the endomorphism of  $L_{\underline 1}$ is taken inside the category of perverse sheaves on $\E_{\nu, d, Q}$.
Now, by Theorem ~\ref{L-fine},  we have 
\[
L_{\underline 1} \simeq \oplus_{\underline \lambda: \underline 1\leq \underline \lambda}  \IC (\E_{ \underline \lambda}) \otimes V_{\underline \lambda},
\]
which implies that 
\[
H_{top} (Y_{\underline 1})  \simeq \oplus_{\underline \lambda}  \End(V_{\underline \lambda}) \simeq  \mbb C[S_{\nu_1} \times \cdots \times S_{\nu_m}].
\]
Therefore, we have 

\begin{thm}
\label{Y}
$H_{top}(Y_{\underline 1}) \simeq \mbb C [ S_{\nu_1}\times  \cdots \times S_{\nu_m}]$, as $\mbb C$-algebras and  the set $\{\overline {Y_M}| M\in S_{\nu_1}\times \cdots \times S_{\nu_m}\}$ is a basis for 
$\mbb C [ S_{\nu_1} \times \cdots \times S_{\nu_m}]$.
\end{thm}

\begin{cor}
Let $\widetilde{\mathcal F}_{\underline 1, x}$ be the fiber of $x$ under the map $\pi_{\underline 1}$. We have 
$H_{top} (\widetilde{\mathcal F}_{\underline 1, x}) \simeq V_{\underline \lambda}$ for any element  $x$ in a relevant stratum of $\pi_{\underline 1}$ which is open dense in $\in \E_{\underline \lambda}$.
\end{cor}

This follows from ~\cite[8.9.14(b)]{CG97}.

\subsection{An analogue of Grothendieck's simultaneous resolution}
\label{Grothendieck}

Similar to the representation space $\E_{\nu, d, Q}$ for $Q$, we define the representation space of $\bar Q$ of dimension vector $(\nu, d)$ to be 
\[
\E_{\nu, d, \bar Q} =\End(V) \times \Hom (D, V).
\]
Elements in $\E_{\nu, d, \bar Q}$ will be represented by $\bar x = (x_{\bar \sigma}, x_{\bar \rho})$. 
An element  $(\bar x, \underline V)\in \E_{\nu, d, \bar Q}\times \Fv$ is called $a$ $quasi$-$stable$ $pair$ if
\begin{itemize}
\item $x_{\bar \sigma} ( V_{i, j}) \subseteq V_{i, j}$ for any $i$ and $j$;

\item $ x_{\bar \rho} ( D_i) \subseteq V_{i+1} $ for any $i$.
\end{itemize}
Let $\tGv$ be the variety of all quasi-stable pairs in $ \E_{\nu, d, \bar Q}\times \Fv$. We have the following diagram
\begin{align}
\label{regular}
\begin{CD}
\Fv @<<< \tGv @>\xi_{\underset{\to}{\nu}} >> \E_{\nu, d,\bar Q},
\end{CD}
\end{align}
where the first map is the projection to the second component and the second map is the projection to the first component. 
(Compare with (\ref{factor}).)
We set
\[
L_{\underset{\to}{\nu}, \bar Q} = (\xi_{\underset{\to}{\nu}})_! (\mbb C_{\tGv})[\dim \tGv].
\]
Then we have
 
\begin{prop}
\label{Fourier}
$\Phi_{Q, \bar Q} (L_{\underset{\to}{\nu}} ) = L_{\underset{\to}{\nu}, \bar Q}$, where $\Phi_{Q, \bar Q}$ is the Fourier-Deligne transform from $\D(\E_{\nu, d, Q}) $ to $\D(\E_{\nu, d, \bar Q})$.
\end{prop}

The proof of this proposition is exactly the same as the proof of \cite[Proposition 10.2.2]{L93}. We leave it to the reader.
Of course, there is a similar result for other orientations of $\Gamma$.

Note that when $d=0$, the map $\xi_{\underset{\to}{\nu}}$ is the generalized Grothendieck simultaneous resolution.  It is well known that the generalized Grothendieck simultaneous resolution is small. 
One can show that $\xi_{\underset{\to}{\nu}}$ is semismall in a similar manner as  the second proof of Proposition ~\ref{pi-small}.
Moreover, 

\begin{prop}
The morphism  $\xi_{\underset{\to}{\nu}}$ in $(\ref{regular})$  is  small.
 \end{prop}
 
 \begin{proof}
We assume that $\underset{\to}{\nu} =\underline 1$. Let 
\[
\mu': \widetilde {\mathfrak g} \to \End(V) 
\]
be the Grothendieck's simultaneous resolution of type $A$, where $\widetilde {\mathfrak g} $ is the variety of quasi-stable pairs $(\underline V, x_{\bar \sigma})$ (with $x_{\bar\rho}=0$)  in $\F_{\underline 1}\times \End(V)$ such that $x_{\bar \sigma}(V_{i,j}) \subseteq V_{i,j}$ for any $i$ and $j$. It is well-known that $\mu'$ is a small resolution with the only relevant strata $\End(V)^{rs}$ consisting of all regular semisimple elements in $\End(V)$.

Let $\E_{\underline 1, \bar Q}$ be the image of the map $\xi_{\underline 1}$. Let $\E_{\underline 1, \bar Q}^{rs}$ be the open subvariety in $\E_{\underline 1, \bar Q}$ consisting of all elements 
whose $\bar\sigma$ components are in $\End(V)^{re}$. By arguing in a similar way as the second proof of Proposition ~\ref{pi-small}, we see that  relevant strata in $\E_{\underline 1}$ with respect to $\xi_{\underline 1}$
are contained in $\E_{\underline 1, \bar Q}^{rs}$.

For any regular semisimple element $x_{\bar \sigma}$, let us fix a basis $\{ u_l| 1\leq l \leq \nu\}$ of $V$ consisting of eigenvectors of $x_{\bar \sigma}$.
Let $\underline V$ be the complete flag whose $l$-th step $V_l$ is the vector subspace spanned by the vectors $u_l, \cdots, u_{\nu}$. 
For any $s\in S_{\nu}$, we denoted by $s\underline V$ be the flag obtained  from $\underline V$ whose $l$-th step is the subspace spanned by the vectors $u_{s(l)}, \cdots, u_{s(\nu)}$.
It is clear that the fiber of $x_{\bar \sigma}$ under $\mu'$ is $\{(sV, x_{\bar \sigma})|s\in S_{\nu}\}$.

Let $\E_{\underline 1, \bar Q}^0$ be the open dense  subvariety  of $\E_{\underline 1,\bar Q}^{re}$ 
consisting of all elements $\bar x=(x_{\bar \sigma},x_{\bar \rho})$ such that   
\[
\# \{u_l| 1\leq l\leq \nu\} \cap x_{\bar \rho} (D_i)  = \nu_{i+1}+\cdots + \nu_m, \quad \forall 1\leq i\leq m, 
\]
where $\{u_l| 1\leq l \leq \nu\}$ is  a basis of $V$ consisting of eigenvectors of 
$x_{\bar \sigma}$ and $\nu_i$ is the part of $\nu$ in the definition of the multi-composition  $\underline 1$. Note that the above condition is independent of the choice of the eigenvectors of $x_{\bar \sigma}$ because
$x_{\bar \rho}(D_i)$ is  a vector subspace. The fact that $\E_{\underline 1,\bar Q}^0$ is open can be proved in the following way. 
We consider the projection $\E_{\underline 1, \bar Q}^0 \to \End(V)^{rs}$. 
We observe that the dimension of $\E^0_{\underline 1, \bar Q}$ is equal to that of  $\E_{\underline 1,\bar Q}^{re}$. 

Since the fiber of $\pi_{\underline 1}$ at any point in  $\E_{\underline 1,\bar Q}^0$ has dimension zero, we see that the relevant strata can only be found in $\E_{\underline 1, \bar Q}^0$.
Consider the restriction $\xi_{\underline 1}^{-1}(\E_{\underline 1, \bar Q}^0 ) \to \E_{\underline 1, \bar Q}^0$  of $\xi_{\underline 1}$ to $\xi_{\underline 1}^{-1}(\E_{\underline 1, \bar Q}^0)$. 
It is a $S_{\nu_1}\times \cdots \times S_{\nu_m}$-principal covering. So $\E_{\underline 1, \bar Q}^0$ is  smooth, hence it is the only relevant stratum for $\xi_{\underline 1}$. 
So we have proved that $\xi_{\underline 1}$ is small.

The statement that $\xi_{\underset{\to}{\nu}}$ is small in general follows from the fact that $\xi_{\underline 1}$ is small  because $\xi_{\underline 1} $ factors through $\xi_{\underset{\to}{\nu}}$.
 \end{proof}
 
 From the above proof, we have the following diagram
 \[
 \begin{CD}
 \xi_{\underline 1}^{-1}(\E_{\underline 1, \bar Q}^0 ) @>\xi_{\underline 1}^0>> \E_{\underline 1, \bar Q}^0\\
 @VVV @VVV \\
 \widetilde{\mathcal G}_{\underline 1} @>\xi_{\underline 1}>> \E_{\underline 1, \bar Q},
 \end{CD}
 \]
 where the vertical maps are open embeddings. Since the morphism $\xi_{\underline 1}^0$ is a $S_{\nu_1}\times \cdots \times S_{\nu_m}$-principal covering, 
 we see (\cite{C98}, ~\cite{KW01}) that 
 $S_{\nu_1}\times \cdots \times S_{\nu_m}$ acts on the local system $( \xi_{\underline 1 })_! (\mbb C_{ \xi_{\underline 1}^{-1}(\E_{\underline 1, \bar Q}^0 )}) $ on $\E_{\underline 1, \bar Q}^0$, and, moreover,
 \[
( \xi^0_{\underline 1 })_! \left (\mbb C_{ \xi_{\underline 1}^{-1}(\E_{\underline 1, \bar Q}^0 )} \right ) \simeq \oplus_{\underline \lambda: \underline 1\leq \underline \lambda}   \mathcal L_{\underline \lambda} \otimes V_{\underline \lambda},
 \]
 where  $V_{\underline \lambda}$ is the simple $S_{\nu_1}\times \cdots \times S_{\nu_m}$-module parametrized by $\underline \lambda$ and 
 $\mathcal L_{\underline \lambda}$ is the irreducible local system on $\E_{\underline 1, \bar Q}^0 $ determined by $V_{\underline \lambda}$.
 Since $\xi_{\underline 1}$ is small, by perverse continuity, we have

 \begin{prop}
 \label{xi-decomposition}
 $L_{\underline 1, \bar Q} \simeq \oplus_{\underline \lambda: \underline 1\leq \underline \lambda}  \IC(\E_{\underline 1, \bar Q},  \mathcal L_{\underline \lambda}) \otimes V_{\underline \lambda}$.
 \end{prop}

We are now ready to prove Theorem ~\ref{L-fine}.  By Propositions ~\ref{Fourier} and ~\ref{xi-decomposition}, we see that the number of non-isomorphic simple perverse sheaves appearing in 
$L_{\underline 1}$ is the same as the number of $\underline \lambda$ such that $\underline 1\leq \underline \lambda$. 
It is clear that $\IC(\E_{\underline \lambda})$ appears in $L_{\underline 1}$. By Lemma  ~\ref{unique}, 
we see that the set $\{ \IC(\E_{\underline \lambda})| \underline 1\leq \underline \lambda\}$  contains all simple perverse sheaves appeared in $L_{\underline 1}$. 
To this end, we see that in order to show Theorem ~\ref{L-fine},  it is enough to show that 
\[
\Phi_{Q, \bar Q} ( \IC(\E_{\underline \lambda}))=
\IC(\E_{\underline 1, \bar Q}, \mathcal L_{\underline \lambda}) , \quad \forall \underline 1\leq \underline \lambda.
\]
This follows from Proposition ~\ref{Fourier} and by induction because 
$\IC(\E_{\underline \lambda})$ and  $\IC(\E_{\underline 1, \bar Q}, \mathcal L_{\underline \lambda})$ 
are the leading terms (see (\ref{leading})) of the complexes $L_{\underline \lambda}$ and $L_{\underline \lambda, \bar Q}$, respectively. 
This finishes the proof of Theorem ~\ref{L-fine}.

\subsection{Singular support}  We set
\[
\E_{\nu, d} =\E_{\nu, d, Q} \oplus \E_{\nu, d, \bar Q},
\]
where $\E_{\nu, d, Q} $ and $\E_{\nu, d, \bar Q}$ are defined in Section ~\ref{Springer} and \ref{Grothendieck}, respectively.

Given any subspace $U$ in $V$ and $(x_{\sigma}, x_{\bar \sigma})$ in $\End(V) \oplus \End(V)$, we denote by $\overline U$  the smallest 
$(x_{\sigma}, x_{\bar \sigma})$-invariant subspace containing $U$ and $\underline U$ the largest $(x_{\sigma}, x_{\bar \sigma})$-invariant subspace contained in $U$.
(Unfortunately, this notation clashes with the notation of flags. But it is clear from the setting.)

Let $\Pi_{\nu, \underline d}$ be the locally closed subvariety of $\E_{\nu, d}$ consisting of all elements $(x_{\sigma}, x_{\rho}, x_{\bar \sigma}, x_{\bar \rho})$ subject to the following conditions:
\begin{align*}
&x_{\sigma}x_{\bar\sigma} -x_{\bar \sigma}x_{\sigma} -x_{\bar\rho} x_{\rho} =0, \\
&\mbox{$x_{\sigma}$ is nilpotent,} \\
&\overline{x_{\bar \rho}(D_i)} \subseteq \underline { x_{\rho}^{-1} ( D_{i+1})}, \quad \forall 1\leq i \leq m,
\end{align*}
where we set $\underline { x_{\rho}^{-1} ( D_{m+1})}=0$. 

\begin{prop} 
\label{singular}
We have the following results.
\begin{itemize}
\item[(a)] The variety $\Pi_{\nu, \underline d}$ is equidimensional, i.e., all irreducible components are of the same dimension.
\item[(b)] The dimension of $\Pi_{\nu, \underline d}$ is $\nu^2+ \nu d$.
\item[(c) ] The irreducible components $T_{\underline \lambda}$  in $\Pi_{\nu, \underline d} $    are parametrized by  the $m$-partitions, where $T_{\underline \lambda}$ is the irreducible component whose projection to $\E_{\nu, d, Q}$ is $\E_{\underline \lambda}$.
\item [(d) ] The singular support  of $\IC(\E_{\underline \lambda})$ is contained in   $
          T_{\underline \lambda} \cup \cup_{\underline \lambda\lneqq \underline \mu } T_{\underline \mu}$ for any $m$-partitions $\underline \lambda$ of $\nu$.
\end{itemize}
\end{prop}

The rest of this section is devoted to prove this theorem. 

Let $\Lambda_{\nu}$ be the variety of commuting pairs $(x_{\sigma}, x_{\bar\sigma})$ and $x_{\sigma}$ is nilpotent. 
It is well-known (see ~\cite{Li10}) that $\Lambda_{\nu}$ is equidimensional 
of dimension $\nu^2$ and whose irreducible components  are parametrized by the set of partitions of $\nu$. We write
$\Lambda_{\nu, \lambda}$ for the irreducible component such that its projection to the $x_{\sigma}$-component 
is the closure of the $\G$-orbit $\mathcal O_{\lambda^{\perp}}$.

We consider the subvariety $\mbf L_{\nu, d}$ of $\E_{\nu, d}$ consisting of all quadruples $(x_{\sigma}, x_{\rho}, x_{\bar \sigma}, x_{\bar \rho})$ such that
\[
x_{\sigma}x_{\bar\sigma} -x_{\bar \sigma}x_{\sigma}=0, 
\quad \mbox{$x_{\sigma}$ is nilpotent}, \quad x_{\bar \rho}=0.
\]
From the definitions, we see that the natural projection from $\mbf L_{\nu, d}$ to $\Lambda_{\nu}$ is a trivial vector bundle. This implies that 
$\mbf L_{\nu, d}$ is equidimensional of dimension $\nu^2+ \nu d$ and its irreducible components 
are of the form $T_{\lambda}$, the preimage of variety $\Lambda_{\nu, \lambda}$.

Let $\mbf L_{\nu, d}^s$ be the open subvariety of $ \mbf L_{\nu, d}$ defined by the following stability  condition:
\[
\underline {x_{\rho}^{-1}(0)} =\{0\}.
\]
The group $\G$ acts  on $\mbf L_{\nu, d}^s$ and its quotient exists, which  is exactly the variety $\widetilde \Sigma$ defined in ~\cite{G96} (with the chosen curve an affine line).
In particular, we have the result that the subvariety $T_{\lambda}^s = T_{\lambda}\cap \mbf L^s_{\nu, d}$ is  an irreducible component of $\mbf L_{\nu, d}^s$ for any partition $\lambda$ of $\nu$ and all irreducible components of $\mbf L_{\nu, d}^s$ are of this form. 
 
 Now  the statements (a)-(c) can be proved in a   similar way as   that of Proposition 4.4.2 in ~\cite{Li11}.
Let  $\Pi_{\nu, \underline d} =\sqcup_{\underline \nu} \Pi_{\nu, \underline d; \underline \nu} $ be a stratification, where 
$\Pi_{\nu, \underline d; \underline \nu} $  is the locally closed subvariety of $\Pi_{\nu, \underline d}$ consists of all elements such that the flag
$\{ \underline {x_{\rho}^{-1}(D_i) }| i=1, \cdots, m+1 \}$ is of type $\underline \nu$ and the union runs over all $m$-compositions $\underline \nu$ of $\nu$.

The assignment $x \mapsto  \{ \underline {x_{\rho}^{-1}(D_i) }| i=1, \cdots, m \}$, for any $x\in \E_{\nu, d}$,  defines a fibration 
\[
p: \Pi_{\nu, \underline d; \underline \nu} \to \F_{\underline \nu}.
\]
We set
\[
\Pi_{\nu, \underline d; \underline V} = p^{-1} (\underline V), \quad \forall \underline V\in \F_{\underline \nu}.
\]
We define a map
\[
\Pi_{\nu, \underline d; \underline V} \to \mbf L_{V/V_2, D/D_2}\times \mbf L^s_{V_2/V_3, D_2/D_3}\times \cdots \times  \mbf L_{V_m, D_m}^s ,
\]
by sending an element in $\Pi_{\nu, \underline d; \underline V}$ to the collection of  the subquotients of $x$ to $V_i/V_{i+1}\oplus D_i/D_{i+1}$ for any $1\leq i\leq m$. 
One can show that this map is a vector bundle of fiber dimension $\sum_{i<j} \nu_i \nu_j + \nu d- \sum_{i=1}^m \nu_i d_i$. 
By combining the above analysis, we see that the statements (a)-(c) follow.

The statement (d) follows from the fact that $\IC(\E_{\underline \lambda})$ is the leading term of the complex $L_{\underline \lambda}$ and $T_{\underline \lambda}$ is the leading
term of the singular support of $L_{\underline \lambda}$.

\subsection{Example} 
In the case when $\underline d= (1, 0)$. 
The results in this Section, such as Theorems ~\ref{L-fine}, ~\ref{Y} and Proposition  ~\ref{singular} 
have been obtained by Achar-Henderson ~\cite{AH08}, and Finkelberg-Ginzburg-Travkin ~\cite{FGT09}. 
Moreover, it was shown by Achar-Henderson, and Travkin (\cite{T09}) independently that the number of  $\G$-orbits in $\E_{\underline 1}$ are finite and  parametrized by the pair of partitions $(\lambda, \mu)$ such that $|\lambda|+|\mu| =\nu$. As a consequence, the intersection complex $\IC(\E_{\underline \lambda})$ is the intersection complex of the 
$\G$ orbit indexed  by the 2-partition $\underline \lambda$. 

Let $\mbf P_d$ be the stabilizer of the fixed flag $\underline D$ in  (\ref{fixedD}). Then the group $\G\times \mbf P_d$ acts on $\E_{\underline 1}$.
For general $\underline d$, the number of $\G\times \mbf P_d$-orbits in $\E_{\underline 1}$ is infinite.  
For example, the following  elements in $\E_{\underline1}$  belong to pairwise distinct orbits. 
\[
(x_{\sigma}, x_{\rho})= \left (
\begin{pmatrix}
0 & 1 & 0\\
0 & 0 & 1\\
0 & 0 & 0 
\end{pmatrix}, 
\begin{pmatrix}
0 & 1 & a \\
1 & 0 & 0
\end{pmatrix} \right ), 
\quad \forall a\in \mbb C.
\]
Because there are infinitely many $\G\times \mbf  P_d$-orbits,  
we are not sure if the intersection complex $\IC(\E_{\underline \lambda})$ is the intersection complex of a certain $\G\times \mbf P_d$-orbit.

\subsection{Tensor product of irreducible representations of $\mathfrak{sl}_N$}

Fix an integer $N$ and a composition $\underline \nu =(\nu_1, \cdots, \nu_m)$. In this section, we only consider 
$m$-composition $\underset{\to}{\nu}$ of the form $  (\underline {\nu_1},\cdots,  \underline {\nu_m}) $ and each composition  having $N$ parts. 
  
Let 
\[
\tF =\sqcup \tFv,
\]
where the union runs over all $m$-compositions under the above assumption. 
Let 
\[
\pi: \tF \to \E_{\nu, d}
\]
be the projection and $\Z_{\underline \nu} =\tF\times_{\E_{\nu, d}} \tF$ the associated fiber product.  
We also denote by $\tF_x$ the fiber of $x\in \E_{\nu, d}$ under $\pi$. 

Let $(\mbb C^N)^{\otimes \nu}$ be the tensor product of $\nu$ copies of the vector space $\mbb C^N$ of dimension $N$. Then the symmetric group $S_{\nu}$ acts 
on $(\mbb C^N)^{\otimes \nu}$ by permuting the components. 
Let  $S(N, \nu) =\End_{S_\nu} ( (\mbb C^N)^{\otimes \nu})$ be the Schur algebra. It is well known that $S(N, \nu)$ is a quotient of the universal enveloping algebra 
$U(\mathfrak {sl}_N)$ of  the  Lie algebra $\mathfrak{sl}_N$.

\begin{prop}
We have 
\[
H_{top} (\Z_{\underline \nu}) \simeq S(N, \nu_1) \otimes S(N, \nu_2) \otimes \cdots \otimes S(N, \nu_m),
\]
as $\mbb C$-algebras. Moreover, the irreducible components in $\Z_{\underline \nu}$ of largest dimension form a basis for $H_{top}(\Z_{\underline \nu})$.
\end{prop}

This is analogous to Theorem ~\ref{Y}. Its proof is similar to that of Theorem ~\ref{Y} by taking into consideration of ~\cite[Section 10]{C98}.

\subsection{Heisenberg action}

Let $\K_{\nu, \underline d}$ be the vector space spanned by the elements in $\mathcal S_{\nu, \underline d}$ in Corollary ~\ref{dimension-S}. 
We set
\[
\K_{\underline d} =\oplus_{\nu\in \mbb N} \K_{\nu, \underline d}. 
\]
Then we have, as $\mbb C$-vector spaces, 
\[
\K_{\underline d} \simeq \mbb F^{\otimes m},
\]
defined by  $L_{\underset{\to}{\nu}} \mapsto x_{\underline{\nu_1}}\otimes \cdots \otimes x_{\underline{\nu_m}}$ for various $\underset{\to}{\nu}$,
where $x_{\underline \nu} = x_{\nu_1} \cdots x_{\nu_n}$ and  $\mbb F^{\otimes m} $ is the tensor product of $m$ copies of Fock space $\mbb F$ (see (\ref{fock})). 
Under such an isomorphism, we have 
\[
\K_{\nu, \underline d} =\bigoplus_{\nu_1+\cdots + \nu_m =\nu}  \mbb F_{\nu_1}\otimes \cdots \otimes  \mbb F_{\nu_m},
\]
where $\mbb F_{\nu}$ is the homogeneous component  of $\mbb F$  of degree $\nu$. Recall that $\mbb F^{\otimes m}$ admits a Heisenberg algebra action. 
We shall give a geometric realization of this action. 
The assignment  $L_{\underline{\nu_1}}\otimes \cdots \otimes L_{\underline{\nu_m}}  \mapsto L_{\underset{\to}{\nu}}$ defines an isomorphism of vector spaces over $\mbb C$:
\[
\phi: \K_{ d_1} \otimes \cdots \otimes \K_{d_m} \to \K_{\underline d}.
\]
On $\K_{ d_1} \otimes \cdots \otimes \K_{d_m} $, there is a geometric Heisenberg action isomorphic to the one on $\mbb F^{\otimes }$. 
We then transport this action to $\K_{\underline d}$ via the isomorphism $\phi$. Geometrically, $\phi$ is nothing but the Hall multiplication (see ~\cite{Li11}) and the action on
$\K_{ d_1} \otimes \cdots \otimes \K_{d_m} $ are defined geometrically by using induction and restriction functors by the argument in Section ~\ref{H1}. 
We see that the resulting action on $\K_{\underline d}$ is geometric, though it is not so natural. 

We finally remark that the obvious induction and restriction functors on $\K_{\underline d}$, for general $\underline d$ of more than two components,  do not produce the Heisenberg action just defined.

\subsection{Acknowledgement}

The main results in this paper are announced in the algebra seminar at Virginia Tech, February 2012 and the Southeastern Lie Theory workshop   on categorification of quantum groups and representation theory at NCSU, April 2012. We thank the organizers for the invitations. We also thank Mark Shimozono for pointing out the interesting problem of computing the intersection cohomology complexes defined in this paper. 
Partial support from NSF grant DMS 1160351 is  acknowledged.

\end{document}